\documentclass{amsart}
\usepackage[centering, includeheadfoot, hmargin=1.0in, tmargin=1.0in, 
  bmargin=1in, headheight=7pt]{geometry}
\usepackage[utf8]{inputenc}
\usepackage{graphicx}
\usepackage{subcaption}
\usepackage{parskip}
\usepackage{enumerate}
\usepackage{amsmath, amsthm, amssymb}
\usepackage{tikz}
\usepackage{float}
\usepackage{filecontents}
\usepackage[alphabetic]{amsrefs}
\usepackage[colorlinks=true,linkcolor=blue,urlcolor=blue, citecolor=blue]%
  {hyperref}

\theoremstyle{plain}
\newtheorem{thm}{Theorem}[section]
\newtheorem{lem}[thm]{Lemma}
\newtheorem{cor}[thm]{Corollary}

\newtheorem{prop}[thm]{Proposition}
\theoremstyle{definition}
\newtheorem{rmk}[thm]{Remark}
\newtheorem{defn}[thm]{Definition}
\newtheorem{ex}[thm]{Example}

\newcommand{\ZZ}{\mathbb{Z}}

\newcommand{\PP}{\mathbb{P}}
\newcommand{\CC}{\mathbb{C}}
\newcommand{\kk}{\ensuremath{\Bbbk}}
\newcommand{\cO}{\mathcal{O}}

\newcommand{\ann}{\operatorname{ann}}

\newcommand{\Supp}{\operatorname{Supp}}
\newcommand{\coker}{\operatorname{coker}}

\newcommand{\myparabola}[6]{
\draw[very thick, color=black] plot [domain=(#1-0.2):(#5+0.5)]
    (   {\x},{(#2*(\x-#3)*(\x-#5)/((#1-#3)*(#1-#5)))+
            (#4*(\x-#1)*(\x-#5)/((#3-#1)*(#3-#5)))+
            (#6*(\x-#1)*(\x-#3)/((#5-#1)*(#5-#3)))} );
}

\begin{document}

\newcommand{\Addresses}{{
  \bigskip

  \noindent Jiyang Gao, \textsc{Department of Mathematics, Massachusetts Institute of Technology, Cambridge, MA 02139}\par\nopagebreak
  \textit{E-mail address}: \texttt{gaojy98@mit.edu}

  \medskip

  \noindent Yutong Li, \textsc{Department of Mathematics and Statistics, Haverford College, Haverford, PA 19041}\par\nopagebreak
  \textit{E-mail address}: \texttt{yli11@haverford.edu}
  
  \medskip

  \noindent Michael C. Loper, \textsc{School of Mathematics, University of Minnesota, Minneapolis, MN 55455}\par\nopagebreak
  \textit{E-mail address}: \texttt{loper012@umn.edu}
  
  \medskip
  
  \noindent Amal Mattoo, \textsc{Department of Mathematics, Harvard College, Cambridge, MA 02138}\par\nopagebreak
  \textit{E-mail address}: \texttt{amattoo@college.harvard.edu}

}}

\title{Virtual Complete Intersections in $\mathbb{P}^1\times \mathbb{P}^1$}
\author{Jiyang Gao, Yutong Li, Michael C. Loper, and Amal Mattoo}

\subjclass[2010]{13D02; 14M25}

\begin{abstract}
  The minimal free resolution of the coordinate ring of a complete intersection in projective space is a Koszul complex on a regular sequence. In the product of projective spaces $\PP^1 \times \PP^1$, we investigate which sets of points have a virtual resolution that is a Koszul complex on a regular sequence. This paper provides conditions on sets of points; some of which guarantee the points have this property, and some of which guarantee the points do not have this property.
\end{abstract}
\maketitle

\vspace{-1cm}
\section{Introduction}\label{sec:intro}
Complete intersections are a fundamental object of study in commutative algebra and algebraic geometry. In projective space $\PP_\kk^r$, a complete intersection $Y$ of codimension $t$ is defined by an ideal of codimension $t$ which can be generated by exactly $t$ elements of the ring $\kk[x_0,\ldots, x_r]$. In this case, there are hypersurfaces $H_1, \ldots, H_t$ such that $Y$ is the scheme-theoretic intersection of the $H_i$'s. Complete intersections have coordinate rings that are Cohen--Macaulay. The defining ideal of a complete intersection in $\PP^r$ is generated by a regular sequence and so the minimal free resolution of the coordinate ring is a Koszul complex \cite{peeva_graded_2011}*{Theorem~14.7}.

Unfortunately, in a product of projective spaces, the nice properties of complete intersections in $\PP^r$ are not completely captured homologically. A zero dimensional scheme $X \subset \PP^1 \times \PP^1$ is a scheme-theoretic complete intersection or virtual complete intersection if there are two polynomials $f$ and $g$ such that the ideal sheaf generated by $f$ and $g$ equals the ideal sheaf of $X$. On the other hand, $X$ is an ideal-theoretic complete intersection (in this paper we will just say complete intersection) if $I_X$, the set of all functions in the Cox ring vanishing on $X$, is generated by two elements. In $\PP^r$, these two notions are the same, but in $\PP^1 \times \PP^1$, they are not (see Example~\ref{2pts} below). This is where the virtual resolutions of \cite{2017arXiv170307631B} help. By allowing some irrelevant homology in a free complex, we expand the notion of a complete intersection via a virtual resolution, while still reaping the benefits of many properties of complete intersections. The goal of this paper is to state conditions on whether a set of points in $\PP^1 \times \PP^1$ do or do not form a virtual complete intersection (see Definition~\ref{def:vci}).

\begin{ex}\label{2pts}
    Consider the zero-dimensional reduced scheme $X \subset \PP^1 \times \PP^1$ consisting of the two points $([1:0],[0:1])$ and $([0:1],[1:0])$. The set of all functions vanishing on $X$ is
    \[I_X = \langle x_0x_1, x_0y_0, x_1y_1, y_0y_1 \rangle.\]
    However, the two polynomials $x_0y_0$ and $x_1y_1$ generate the same ideal sheaf as $I_X$ does. Therefore $X$ is a scheme-theoretic complete intersection or virtual complete intersection, but not an ideal-theoretic complete intersection.
\end{ex}

Points in $\PP^1 \times \PP^1$ have been studied in the past, but often from a point of view of studying the saturated defining ideals of these points in the Cox ring of $\PP^1 \times \PP^1$. Some results include several classifications of when both reduced and fat points in $\PP^1 \times \PP^1$ are arithmetically Cohen--Macaulay \cites{GMR, GMR-resolutions, Guardo-FatPoints, GVT08, GVT-separators, GVT-FatPoints, GVT-Minres, GVT11}. Further characterizations of points in more general products of projective spaces can be found in \cites{GVT08, GVT-separators, FGM, GVT-Minres, GVT11, VT} . In \cite{GMR}, Giuffrida, Maggioni, and Ragusa prove that points in $\PP^1 \times \PP^1$ are defined by the ideal generated by two forms of bidegree $(a,0)$ and $(0,b)$, and further, if $f$ and $g$ are two forms of any bidegree in $\PP^1 \times \PP^1$, then the ideal $\langle f, g \rangle$ is not saturated, except in this case. In this paper, we study when points have a complete intersection ideal that saturates to the defining ideal of the set of points, which is equivalent to being a virtual complete intersection. It turns out that all sets of points that are virtual complete intersections are not arithmetically Cohen--Macaulay, with the exception of points that are complete intersections (Corollary~\ref{ferrers}). While the results in this paper concentrate on points in $\PP^1 \times \PP^1$, perhaps recent results of \cite{FGM}, which uses techniques of liaison, could help find VCIs in any product of projective spaces.

\subsection*{Setup}\label{sec:notation}
Let $\kk$ be an algebraically closed field. In this paper we are concerned mostly with reduced zero-dimensional schemes in the product of projective spaces $\PP^1 \times \PP^1$ over $\kk$. The Cox ring of $\PP^1 \times \PP^1$ is the $\ZZ^2$-graded ring $S := \kk[x_0,x_1,y_0,y_1]$ where $\deg(x_i) = (1,0)$ and $\deg(y_i)=(0,1)$. The irrelevant ideal of $S$ is $B := \langle x_0,x_1 \rangle \cap \langle y_0, y_1 \rangle = \langle x_0y_0, x_0y_1, x_1y_0,x_1y_1 \rangle$. In this setting, closed subschemes are in one-to-one correspondence with $B$-saturated bihomogeneous ideals \cite{CLS}*{Proposition~6.A.7}. The $B$-saturation of an ideal $I$ is
\[I:B^\infty = \bigcup_{k \ge 0} I:B^k = \{s \in S | sB^k \subset I \text{ for some } k\}.\]
If $I \subset S$ is an ideal, then $V(I)$ denotes the subscheme of $X$ consisting of all $B$-saturated bihomogeneous prime ideals that contain $I$. On the other hand, if $X$ is a subvariety of $\PP^1 \times \PP^1$, then $I_X$ denotes the $B$-saturated bihomogenous ideal of polynomials in $S$ that vanish at every point in $X$.

We will call reduced subschemes of $\PP^1 \times \PP^1$ ``sets of points in $\PP^1 \times \PP^1$.'' The maximum number of points on a single horizontal ruling in a set of points $X$ is denoted as $m$, and the maximum number of points on a single vertical ruling is denoted as $n$.

\begin{defn}[{\cite[Definition~1.1]{2017arXiv170307631B}}]
A \textbf{\emph{virtual resolution}} for a module $S/I$ in the biprojective space $\mathbb{P}^1\times \mathbb{P}^1$ is a $\mathbb{Z}^2$-graded complex of free $S$-modules
$$F_\bullet := [F_0 \overset{\varphi_1}{\longleftarrow} F_1\overset{\varphi_2}{\longleftarrow} F_2\overset{\varphi_3}{\longleftarrow}\cdots]$$
such that for every homology modules $H_i(F_\bullet)$, it is true that $\ann (H_i(F))\supseteq B^\ell$ for some $\ell>0$, and the associated sheaves $\widetilde{S/I}$ and $\widetilde{\coker(\varphi_1)}$ are isomorphic.
\end{defn}

\begin{defn}\label{def:vci} Let $X$ be a set of points in $\PP^1 \times \PP^1$ with defining ideal $I_X$. We say $X$ is a \textbf{\emph{virtual complete intersection (VCI)}} if $S/I_X$ has a virtual resolution that is a Koszul complex $K(f,g)$ of bihomogeneous forms $f$ and $g$, where bihomogenous means every term in the polynomial has the same $x$-degree and $y$-degree.
\end{defn}

In other words, a set of points $X \subset \PP^1 \times \PP^1$ is called a virtual complete intersection or scheme-theoretic complete intersection generated by 2 forms $f$, $g \in S$ with $\deg(f) = (a,b)$ and $\deg(g) = (c,d)$ if there exists a sheaf surjection
\[0 \longleftarrow \mathcal{I}_X \longleftarrow \cO_{\PP^1 \times \PP^1}(-a,-b) \oplus \cO_{\PP^1 \times \PP^1} (-c, -d).\]
Given two curves of $\PP^1 \times \PP^1$ having no common component, $C$ of bidegree $(a,b)$ defined by equation $f = 0$, and $D$ of bidegree $(c,d)$ of equation $g = 0$, let $X = C \cap D$ be their scheme-theoretic complete intersection. The ideal $\langle f,g \rangle \subset S$ is not saturated (except in the cases $b = c = 0$ or $a = d = 0$), but we have the exact (Koszul) sequence of sheaves
\[ 0 \longleftarrow \mathcal{I}_X \longleftarrow \cO(-a,-b) \oplus \cO(-c,-d) \longleftarrow \cO(-a-c,-b-d) \longleftarrow 0.\]

Next, we review the notion of configurations as introduced in \cite[\textsection 3.2]{guardo_arithmetically_2015} and show that the property of being a VCI is not a combinatorial invariant. Points in $\PP^1 \times \PP^1$ may be placed on a grid, according to their coordinates in each copy of $\PP^1$, in the following way. There are two projections $\pi_i \colon \PP^1 \times \PP^1 \to \PP^1$:
$$\pi_1(a,b) = a \quad \text{ and } \quad \pi_2(a,b) = b.$$
Making a grid of horizontal and vertical lines, the vertical lines correspond to the first copy of $\PP^1$ and the horizontal lines correspond to the second copy of $\PP^1$. Two points $p, q \in \PP^1 \times \PP^1$ lie on the same vertical line if $\pi_1(p) = \pi_1(q)$. They lie on the same horizontal line if $\pi_2(p) = \pi_2(q)$. By permuting the horizontal and vertical lines, we arrange the points so that the number of points on each horizontal and vertical line decreases from top to bottom and from left to right, forming a \textbf{\emph{configuration}}.

For example, letting $a_i$ denote a point in the first copy of $\PP^1$ and $b_i$ denote a point in the second copy of $\PP^1$, the set of points
\[\{(a_1,b_1), (a_2, b_1), (a_3, b_1), (a_1,b_2), (a_4, b_2), (a_2,b_3),(a_5,b_4)\}\]
in $\PP^1 \times \PP^1$, can be represented as in Figure~\ref{config} (here the points are labeled, but in what follows they will not be labeled). Note that the configuration of a set of points is not unique: in Figure~\ref{config}, switching the horizontal rulings with coordinates $b_2$ and $b_3$ also yields a valid configuration. Thus, we consider two sets of points in $\PP^1\times \PP^1$ to be \textbf{\emph{equivalent up to configuration}} if they have the same configurations after permutation and relabeling of the rulings.

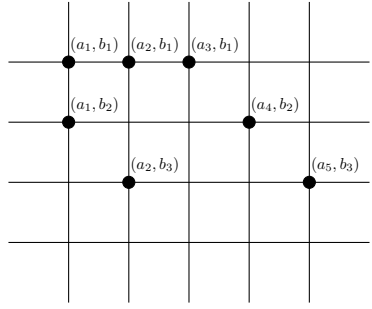
\begin{figure}
\begin{center}
  \begin{tikzpicture}[scale = 0.8]
    \draw (-1,0) -- (5,0);
    \draw (-1,1) -- (5,1);
    \draw (-1,2) -- (5,2);
    \draw (-1,3) -- (5,3);
    \draw (0,-1) -- (0,4);
    \draw (1,-1) -- (1,4);
    \draw (2,-1) -- (2,4);
    \draw (3,-1) -- (3,4);
    \draw (4,-1) -- (4,4);
    \draw[fill] (0,3) circle [radius=0.1];
    \node [above right] at (0,3) {\scalebox{0.6}{\hspace{-.17cm}$(a_1,b_1)$}};
    \draw[fill] (1,3) circle [radius=0.1];
    \node [above right] at (1,3) {\scalebox{0.6}{\hspace{-.17cm}$(a_2,b_1)$}};
    \draw[fill] (2,3) circle [radius=0.1];
    \node [above right] at (2,3) {\scalebox{0.6}{\hspace{-.17cm}$(a_3,b_1)$}};
    \draw[fill] (0,2) circle [radius=0.1];
    \node [above right] at (0,2) {\scalebox{0.6}{\hspace{-.17cm}$(a_1,b_2)$}};
    \draw[fill] (3,2) circle [radius=0.1];
    \node [above right] at (3,2) {\scalebox{0.6}{\hspace{-.17cm}$(a_4,b_2)$}};
    \draw[fill] (1,1) circle [radius=0.1];
    \node [above right] at (1,1) {\scalebox{0.6}{\hspace{-.17cm}$(a_2,b_3)$}};
    \draw[fill] (4,1) circle [radius=0.1];
    \node [above right] at (4,1) {\scalebox{0.6}{\hspace{-.17cm}$(a_5,b_3)$}};
  \end{tikzpicture}
\end{center}
\caption{An example of a configuration}
\label{config}
\end{figure}

Unfortunately, the property of being a VCI depends on the coordinates of the points, not just on their configuration. This is not so surprising as the Betti numbers of points in $\PP^1 \times \PP^1$ also depend on more than just the configuration. To illustrate this point, we use the cross ratio.

\begin{defn}[{\cite[Section 3.2 Definition 12]{ahlfors}}]
If four points in $\PP^1$ have homogeneous coordinates $[a:a'],[b:b'],[c:c'],[d:d']$, their \textbf{\emph{cross ratio}} is:
$$\frac{(ca'-ac')(db'-bd')}{(da'-ad')(cb'-bc')}.$$
\end{defn}
If a point is $[1:0]$ or $[0:1]$, then the terms involving this point are dropped from both the numerator and the denominator.\par 

\begin{figure}[H]
    \centering
    \begin{tikzpicture}[scale=.8]
    \foreach \x in {1,2,3,4} {
    	\draw[thin] (\x, 0.5) -- (\x, 4.5);
    	\draw[thin] (0.5, \x) -- (4.5, \x);}
    \coordinate (a) at (1,4);
    \coordinate (b) at (2,3);
    \coordinate (c) at (3,2);
    \coordinate (d) at (4,1);
    \foreach \v in {a,b,c,d}
	\fill (\v) circle (3pt);
    \end{tikzpicture}
    \caption{A four-point configuration whose the minimal free resolution depends on the coordinates.}
    \label{4-pts}
\end{figure}
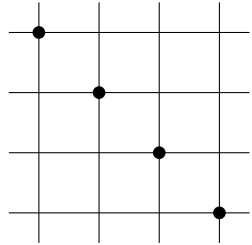
\par In Figure \ref{4-pts}, the total Betti numbers of the minimal free resolution depends on the cross ratio. Let $I$ be the ideal of bihomogeneous forms vanishing at the points. When the cross ratios of the coordinates are the same after projection to each copy of $\PP^1$, the minimal free resolution of $S/I$ (omitting the twists of the free modules) is
\begin{equation*}
    S^1\longleftarrow S^6\longleftarrow S^8\longleftarrow S^3\longleftarrow 0. 
\end{equation*}
When the cross ratios of the two copies of $\PP^1$ are different, the minimal free resolution is
\begin{equation*}
    S^1\longleftarrow S^6\longleftarrow S^7\longleftarrow S^2\longleftarrow 0.
\end{equation*}
Moreover, for any collection of points with a subconfiguration of this kind, the minimal free resolution will depend on the value of the coordinates. By contrast, this configuration is always a VCI, regardless of the cross ratios (by Theorem \ref{same-mult}).

\begin{prop}
Given the configuration of four points in Figure~\ref{4-pts}, the minimal resolution of these points depends on whether or not the cross ratios are equal after projection to each copy of $\PP^1$. 
\end{prop}

\begin{proof}
We may change coordinates so that three of the four points are $[0:1],[1:1],[1:0]$ and the last point is $[1:c]$, where $c$ is the cross ratio \cite[Section 3.2 Definition 12]{ahlfors}. Now consider the form $x_0y_1-x_1y_0$. If the cross ratios on both copies of $\PP^1$ are the same, the form $x_0y_1-x_1y_0$ vanishes, which explains the degree differences in the minimal free resolutions in the two cases - one has a $(1,1)$ graded piece whereas the other has two $(1,2)$ forms. Since the Hilbert function is recoverable from the minimal free resolution, the minimal free resolution changes accordingly.
\end{proof}

Virtual resolutions are also not invariant under configurations when the total number of points is large relative to the maximum number of points lying on a single horizontal ruling and a single vertical ruling, which we respectively denote by $m$ and $n$.

\begin{rmk}\label{hyperbola}
When $|X|\ge mn$, VCIs are not always determined by configuration. That is, the same configuration may be a VCI with some coordinates, but not with others. For example, the configuration below is a VCI when either the four rightmost points or the four bottommost points lie on a $(1,1)$-form. If these points do not lie on such a conic, the configuration is not a VCI. This example is explored in more detail in Theorem~\ref{badconfiguration}, after the necessary machinery has been developed.
\end{rmk}

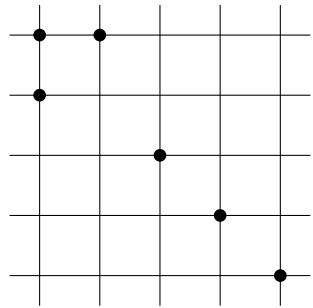
\begin{figure}[H]
  \centering
  \begin{tikzpicture}[scale=.8]
\foreach \x in {2,3,4,5,6}
	\draw[thin] (\x, 1.5) -- (\x, 6.5);
\foreach \y in {2,3,4,5,6}
	\draw[thin] (1.5, \y) -- (6.5, \y);
\coordinate (a) at (2,5);
\coordinate (b) at (2,6);
\coordinate (c) at (3,6);
\coordinate (d) at (4,4);
\coordinate (e) at (5,3);
\coordinate (f) at (6,2);
\foreach \v in {a,b,c,d,e,f}
	\fill (\v) circle (3pt);
  \end{tikzpicture}
  \caption{Whether this configuration is a VCI depends on the coordinates of points.}
  \label{fig:badconfig}
\end{figure}

However, the configuration above is far from being an Arithmetically Cohen--Macaulay set of points or a complete intersection in \cite[Theorem 4.11, Theorem 5.13]{guardo_arithmetically_2015}, whose criteria depend only on the combinatorial configuration and not the actual coordinates of the points. Hence the question of when sets of points form VCIs is another interesting and subtle question.

\subsection*{Summary of Main Results}
Complete intersections are always VCI, however, VCIs form a strictly larger set of points than complete intersections. Our main results are summarized below.

\begin{enumerate}
    \item A set of points is a VCI when it has the same number of points in each vertical (or each horizontal) ruling (Theorem \ref{same-mult}).
    \item A set of points $X$ is not a VCI when
    \begin{enumerate}
    \item $|X|<mn$, and there is at least one point in $X$ that is on a horizontal ruling with $m$ points and a vertical ruling with $n$ points (Theorem \ref{cross}).
    \item $|X|<mn$ and $\gcd(m,n)$ does not divide $|X|$ (Theorem \ref{gcd-thm}).
    \item The degrees of two forms that intersect at $X$ are known and one of the conditions in Proposition~\ref{num-theory-prop} holds.
    \end{enumerate}
    \item VCIs are not solely determined by the configuration of the points, which is a characterization of where points lie in relation to each other: when $|X| > mn$, the actual coordinates of points can play a role in determining whether or not $X$ is a VCI (Remark~\ref{hyperbola}, Theorem~\ref{badconfiguration}). 
    \item When all points lie on at most three vertical or horizontal rulings, we provide a complete classification of VCIs (Section~\ref{sec-compl-cla}).
\end{enumerate}

\begin{ex}\label{3ptVCI} Consider Figure \ref{VCI-not-CI} of the three points
\[X = \{([1:0],[0:1]),([1:1],[0:1]),([0:1],[1:1])\} \subset \PP^1 \times \PP^1.\]
  
Letting $I_X$ be the $B$-saturated ideal of bihomogenous polynomials vanishing at $X$, the minimal free resolution of $S/I_X$ is
\[0 \longleftarrow S \longleftarrow \begin{matrix} S(0,-2) \\ \oplus \\ S(-1,-1) \\ \oplus \\ S(-2,-1) \\ \oplus \\ S(-3,0) \end{matrix}
\longleftarrow \begin{matrix} S(-1,-2) \\ \oplus \\ S(-2,-2) \\ \oplus \\ S(-3,-1)^2 \end{matrix}
\longleftarrow S(-3,-2) 
\longleftarrow 0.\]
It is well known that $X$ is
not Arithmetically Cohen--Macaulay by the criterion in \cite[Theorem~4.11]{guardo_arithmetically_2015}. Therefore $X$ does not form a complete intersection. As the picture indicates, however, $X$ is a VCI as it is the intersection of the varieties of two forms
\[f = x_1y_1 \hspace{0.5cm} \text{and} \hspace{0.5cm} g = x_0(x_1 - x_0)(y_1 - y_0).\]
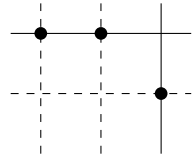
\begin{figure}[H]
    \centering
    \begin{tikzpicture}[scale=.8]
    \draw[thin] (6.5, 3) -- (9.5,3);
    \draw[thin] (9, 1) -- (9,3.5);
    \draw[thin, dashed] (6.5, 2) -- (9.5,2);
    \draw[thin, dashed] (7, 1) -- (7,3.5);
    \draw[thin, dashed] (8, 1) -- (8,3.5);
    \coordinate (d) at (7,3);
    \coordinate (e) at (8,3);
    \coordinate (f) at (9,2);
    \foreach \v in {d,e,f}
	\fill (\v) circle (3pt);
    \end{tikzpicture}
    \caption{A $3$-point variety that is generated by two forms geometrically but not a complete intersection.}
    \label{VCI-not-CI}
\end{figure}

Therefore, $K(f,g)$ is a virtual resolution of $S/I_X$, where

\[K(f,g) :=  [0 \longleftarrow S \xleftarrow{\begin{bmatrix} f & g \end{bmatrix}} \begin{matrix} S(-1,-1) \\ \oplus \\ S(-2,-1) \end{matrix}
  \xleftarrow{\begin{bmatrix} -g \\ f \end{bmatrix}} S(-3,-2)
  \longleftarrow 0].\]

Although the points do not form a complete intersection, they nonetheless share similar properties with complete intersections. The saturation of $\langle f, g \rangle$ by the irrelevant ideal $B$ is equal to $I_X$ so $V(f) \cap V(g) = X$ scheme-theoretically.
\end{ex}

All theorems proved in this paper are from the virtual viewpoint. That is we are looking for ideals generated by two forms that saturate to the defining ideal of points. This is a slightly different problem than the one studied in \cite{guardo_arithmetically_2015}, which is concerned with when sets of points $X \subset \PP^1 \times \PP^1$ have $B$-saturated ideals $I_X$ so that $S/I_X$ is Cohen--Macaulay. By the Auslander--Buchsbaum formula, this occurs exactly when the minimal free resolution of $S/I_X$ is of length 2. It could be asked when sets of points are ``virtually arithematically Cohen--Macaulay.'' This question has already been answered: all sets of points in $\PP^1 \times \PP^1$ have virtual resolutions of length 2 \cite[Theorem~1.5]{2017arXiv170307631B}. As such in this paper, we concentrate only on when sets of points are VCIs.

\subsection*{Outline}
In Section~\ref{sec-set-theoretic}, it is proved that every set of points forms a VCI when considered set-theoretically. Then, in Sections~\ref{sec-main-results}  and ~\ref{sec-nt-bounds}, we examine the scheme-theoretic case of reduced points. The majority of the proofs of our main theorems are in these sections. We name many conditions which are guaranteed to either give rise to or never give rise to VCIs. Finally, Section \ref{sec-compl-cla} is an application of these results, giving a complete classification of VCIs that lie on at most three horizontal (or vertical) rulings.

\subsection*{Acknowledgements}
\par We are grateful to Christine Berkesch for her extensive guidance and feedback throughout this project, as well as Vic Reiner for many helpful discussions. We would also like to thank the University of Minnesota--Twin Cities for organizing and supporting the combinatorics REU in 2018. The use of \emph{Macaulay2} \cite{M2} was instrumental in creating examples and testing conjectures throughout work on this project. Finally, we would like to thank two anonymous referees for notes on improving the readability of this paper and for identifying related references. This work was funded by NSF RTG grant DMS-1745638.

\section{Set-Theoretic VCIs}\label{sec-set-theoretic}
In this section, we consider zero-dimensional subschemes of $\PP^1 \times \PP^1$. We show that set-theoretically, all such subschemes (hence all corresponding configurations) are virtual complete intersections.
\begin{thm}\label{set-theoretic}
  For any zero-dimensional subscheme $X$ of $\PP^1 \times \PP^1$, there is an ideal $J$ so that $\sqrt{J} = \sqrt{I_X}$ and $S/J$ has a virtual resolution that is a length two Koszul complex.
\end{thm}

\begin{proof}
  Let $\Supp(X)$ be the underlying set of points in $\PP^1 \times \PP^1$. We will show that there is an ideal generated by two bihomogenous forms $f$ and $g$ so that $\Supp(V(f, g)) = \Supp(X)$. 
  
  Let $k$ be the number of distinct vertical rulings containing points of $\Supp(X)$, and let $\ell_i$ be the $(1,0)$-form defining the $i^{th}$ of these vertical rulings. Set $f=\ell_1 \cdots \ell_k$, so $V(f)$ is the smallest union of vertical rulings containing the set of points.
  
  Next, let $n$ be the maximum number of points of $\Supp(X)$ that lie on a single vertical ruling. For each vertical ruling $i$, let $g_i$ be a $(0,n)$ form that is the the product of $n$ $(0,1)$-forms such that $\Supp(X\cap V(\ell_i))=\Supp(V(g_i)\cap V(\ell_i))$. That is, we multiply the forms defining horizontal rulings containing points of the set on $V(\ell_i)$, and repeat some of them so that the degree is $(0,n)$. Set $g=\sum_{i=1}^{k}\frac{fg_i}{l_i}$.
  
  Now we check that $\Supp(X\cap V(\ell_i))=\Supp(V(g)\cap V(\ell_i))$. Let $p$ be contained in the left hand side, which we saw is equal to $\Supp(V(g_i)\cap V(\ell_i))$. We have $g_i(p)=0$ and $\frac{f}{\ell_j}(p)=0$ for $j\neq i$, so $g(p)=0$. This shows $\Supp(X\cap V(\ell_i))\subseteq\Supp(V(g)\cap V(\ell_i))$. 
  
  To show $\Supp(X\cap V(\ell_i))\supseteq\Supp(V(g)\cap V(\ell_i))$, suppose $q$ is contained in the right hand side. Then $g(q)=0$ forces $g_i(q)=0$ since $\frac{f}{\ell_j}$ vanishes if and only if $i\neq j$. Therefore $q \in \Supp(V(g_i)\cap V(\ell_i)) = \Supp(X\cap V(\ell_i))$, and thus $\Supp(X\cap V(\ell_i))=\Supp(V(g)\cap V(\ell_i))$.
  
  Taking the union over all $i$ yields $\Supp(X)=\Supp(V(f)\cap V(g))=\Supp(V(f,g))$ as desired.
  
\end{proof}

\begin{ex}
\label{ex:set-theoretic}
We demonstrate the procedure described in the proof of Theorem~\ref{set-theoretic} to show that the configuration of six points in Figure~\ref{fig:fatpoints} is a set-theoretic virtual complete intersection.

Set
\[\ell_1 = x_0, \hspace{1cm} \ell_2 = x_0 - x_1, \hspace{1cm} \ell_3 = x_0 - 2x_1.\]
We can then choose the $g_i$ to be
\begin{align*}
    g_1 &= (x_1 - x_0)(x_0 - 2x_1)y_0^3, \\
    g_2 &= x_0(x_0 - 2x_1)y_0(y_0 - y_1)^2, \\
    g_3 &= x_0(x_0-x_1)y_0(y_0 - y_1)(y_0-2y_1).
\end{align*}
Letting $f=\ell_1\ell_2\ell_3$ and $g = g_1 + g_2 + g_3$, then $V(f,g)$ is exactly the 6 points in the figure, however, $([0:1],[0:1])$ is counted with multiplicity 2 and $([1:1],[1:1])$ is counted with multiplicity 3.
\end{ex}

\begin{rmk}\label{vis}
Notice that this procedure makes the sum of the multiplicities of points on each horizontal ruling equal. Perhaps an easier way to visualize that this configuration is a set theoretic VCI is visualizing the points with the above multiplicities as the scheme-theoretic intersection of the dotted curves indicated in Figure~\ref{fig:fatpoints},
$$V(x_0(x_0-x_1)(x_0-2x_1))$$
which is the vanishing set of a $(3,0)$ form, and the dashed curves,
$$V((y_0)(2x_1^2y_0+x_0^2y_1-3x_0x_1y_1)(x_0y_1-x_1y_0))$$
which is the vanishing set of a $(3,3)$ form. By \cite[\textsection 4.2.1]{shafarevich_basic_2013} (see Theorem~\ref{bezout}), the intersection should have order $9$, but the curves intersect at $([0:1],[0:1])$ with multiplicity $3$ and at $([1:1],[1:1])$ with multiplicity $2$, so the intersection set is indeed $6$ points. 
\begin{figure}
    \centering
    \begin{tikzpicture} 
        \foreach \x in {0,1,2}
        	\draw[thin] (\x, -0.5) -- (\x,2.5);
        \foreach \y in {0,1,2}
        	\draw[thin] (-0.5, \y) -- (2.5, \y);
        \coordinate (a) at (0,0);
        \coordinate (b) at (1,0);
        \coordinate (c) at (1,1);
        \coordinate (d) at (2,0);
        \coordinate (e) at (2,1);
        \coordinate (f) at (2,2);
        \draw[very thick, densely dotted] (0, -0.5) -- (0, 2.5);
        \draw[very thick, densely dotted] (1, -0.5) -- (1, 2.5);
        \draw[very thick, densely dotted] (2, -0.5) -- (2, 2.5);
        \draw[ultra thick, dashed] (-0.5, 0) -- (2.5, 0);
        \draw[ultra thick, dashed] plot [domain=-0.5:2.5]
        ({\x},{\x});
        \draw[ultra thick, dashed] plot [domain=-0.25:2.5]
        ({\x},{-0.5*(\x)^2+1.5*(\x)});
        \foreach \v in {a,b,c,d,e,f}
        	\fill (\v) circle (3.5pt);
        \node[anchor=west, inner sep=2pt] at (2.5, 0) {\scriptsize{\([0:1]\)}};
        \node[anchor=west, inner sep=2pt] at (2.5, 1) {\scriptsize{\([1:1]\)}};
        \node[anchor=west, inner sep=2pt] at (2.5, 2) {\scriptsize{\([2:1]\)}};
         \node[anchor=north, inner sep=2pt] at (0, -0.5) {\scriptsize{\([0:1]\)}};
        \node[anchor=north, inner sep=2pt] at (1, -0.5) {\scriptsize{\([1:1]\)}};
        \node[anchor=north, inner sep=2pt] at (2, -0.5) {\scriptsize{\([2:1]\)}};
    \end{tikzpicture}
    \caption{Configuration from Example \ref{ex:set-theoretic}, forms from Remark \ref{vis}}
    \label{fig:fatpoints}
\end{figure}
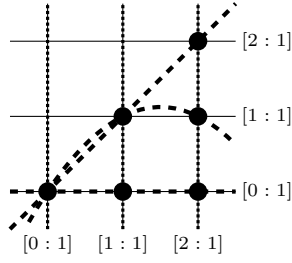
\end{rmk}

\section{Determination of VCIs}\label{sec-main-results}
We now consider reduced zero-dimensional subschemes of $\PP^1 \times \PP^1$, which we refer to as ``sets of points.'' This requires that the homogeneous ideal generated by the two forms equal $I_X$ after saturation by $B$, instead of first taking the radical and then saturating by $B$, which leads to a richer classification of configurations into VCIs, non-VCIs, and coordinate dependent cases.

In the previous section, we proved that set-theoretically all configurations are VCIs by assigning multiplicities so that along each ruling, there are the same total multiplicity of points. When this condition is satisfied without having to artificially ``boost up" the multiplicity of any point, we have a natural environment for VCIs.

\begin{thm}\label{same-mult}
If $X$ has the same number of points in each vertical (or each horizontal) ruling, it is a VCI.
\end{thm}
\begin{proof}
By symmetry, it is enough to prove the vertical case. The proof is nearly identical to the proof of Theorem \ref{set-theoretic}. In the notation of that proof, we construct $f$ and $g$ as before, noting that each $V(l_i)$ contains $n$ points of $\Supp(X)$, so $V(g_i)\cap V(l_i)$ is supported on $n$ distinct points each having multiplicity $1$. Therefore, the computations in the set-theoretic case of the support are exactly the same as in the scheme-theoretic case, showing that $X=V(f)\cap V(g)=V(f,g)$, and thus that $X$ is a VCI.
\end{proof}

Notice that Example~\ref{3ptVCI} illustrates this theorem. On the other hand, Theorem~\ref{cross} below names conditions that guarantee when a set of points can never be a VCI.

\begin{thm}\label{cross}
Let $X$ be a set of points in $\PP^1\times \PP^1$. Let $m$ be the maximum number of points of $X$ on a single horizontal ruling, and let $n$ be the maximum number of points on a single vertical ruling. If $|X|<mn$, and there is at least one point in $X$ that is on a horizontal ruling with $m$ points and a vertical ruling with $n$ points, then $X$ is not a VCI.  
\end{thm}

Before proving Theorem~\ref{cross}, we introduce the generalized B\'ezout's theorem and two technical lemmas to serve as tools for providing bounds on multidegrees. 

\begin{thm}[Generalized B\'ezout's Theorem, {\cite[\textsection 4.2.1]{shafarevich_basic_2013}}]\label{bezout} 
Let $f,g\in S$ be two bihomogeneous forms in $\PP^1\times \PP^1$. If $f$ and $g$ of multidegree $(a,b)$ and $(c,d)$ respectively, are in general position, i.e., $f, g$ have no common factor, then $|V(f)\cap V(g)|=ad+bc$, counting multiplicities.
\end{thm}
This theorem will be used extensively to help combinatorially determine virtual complete intersections.

\begin{lem}\label{max-deg}
Given a configuration of finitely many points $X$ in $\PP^1 \times \PP^1$, let $m$ be the maximum number of points on a single horizontal ruling and $n$ be the maximum number of points on a single vertical ruling. If $K(f,g)$ is a virtual resolution of $S/I_X$, where polynomials $f$ and $g$ are of degrees $(a,b)$ and $(c,d)$, respectively, then $\max(a, c) \ge m$ and $\max(b,d) \ge n$.
\end{lem}
\begin{proof}
Assume, for the sake of contradiction, that both $a$ and $c$ are less than $m$. Without loss of generality, we can change coordinates to assume that the $m$ points are on the horizontal ruling with coordinates $[1:0]$ and assume none of the $m$ points lie on the vertical ruling $[0:1]$. We can restrict $f$ to the horizontal ruling $[1:0]$ by substituting $y_0=1, y_1=0$, $x_0 = 1$ yielding a single variable polynomial of degree $a$ with $m$ roots. By the assumption that $a<m$, this restriction of $m$ must be identically 0, and so $V(f)$ contains the entire ruling $[1:0]$. By an identical argument on $g$ using $c<m$, we have $V(g)$ also containing the entire ruling $[1:0]$. Therefore, $V(f)\cap V(g)$ contains that entire ruling, and so cannot be the original finite set of points. Thus our assumption that both $a$ and $c$ are less than $m$ was false, and so $\max(a, c)\geq m$. The proof that $\max(b, d)\geq n$ is analogous.  
\end{proof}

\begin{lem}\label{case-1}
Let $X$ be a set of points in $\PP^1\times \PP^1$. Let $m$ and $n$ be as in Lemma~\ref{max-deg}. If $K(f,g)$ is a virtual resolution for $S/I_X$, where $f$ and $g$ have multidegrees $(a,b)$ and $(c,d)$ respectively, and $|X|<mn$, then,
\begin{enumerate}
    \item\label{case-1-1} Either (i) $a\ge m$ and $b\ge n$, or (ii) $c\ge m$ and $d\ge n$. 
    \item In case $(i)$, $V(g)$ contains the horizontal ruling containing the $m$ points, and the vertical ruling on the containing the $n$ points. In case (ii), the same is true of $f$.
\end{enumerate}
\end{lem}
\begin{proof}
By Lemma \ref{max-deg}, we have
\[\max(a, c)\geq m \; \text{ and } \; \max(b, d)\geq n.\]
Without loss of generality, suppose $a\ge c$ and $d\ge b$. Then, $a\ge m$ and  $d\ge n$. However, in this case $ad \ge mn$, so $ad+bc\ge mn$, which contradicts $|X|<mn$. Therefore, we must have $a\ge c$ and $b\ge d$, so $a\ge m$ and $b\ge n$. This proves (\ref{case-1-1}).
\par If $g$ does not contain the entire line of the $m$ collinear points, then $g$ restricted to that line is a nonzero polynomial with $m$ roots, and so has degree at least $m$. This means that $c\ge m$, which gives the contradiction $|X|=ad+bc\ge bc\ge mn$. Similarly, if $g$ does not contain the ruling with $n$ points, then its restriction to that line must have degree at least $n$ giving the contradiction $|X|=ad+bc\ge ad\ge mn$. This completes the proof.  
\end{proof}

\begin{proof}[Proof of Theorem~\ref{cross}]
Assume that $|X|<mn$, and there is at least one point in $X$ that lies both on a horizontal ruling with $m$ points and a vertical ruling with $n$ points. If $K(f,g)$ is a virtual resolution for $S/I_X$, where $f$ has multidegree $(a,b)$ and $g$ has multidegree $(c,d)$, then by the first part of Lemma~\ref{case-1}, we may assume $a\ge m$ and $b\ge n$. Suppose $V(g)$ includes $s$ horizontal lines and $t$ vertical lines. Now using the second part of Lemma~\ref{max-deg}, $s$ and $t$ are at least one, and by assumption, the intersection of these $s+t$ lines contains at least one point of $X$. Factoring $g$, and changing coordinates if necessary so that no points have coordinate $[0:1]$, yields 
$$g = \lambda h_0 g_0, \text{where } h_0=(x_1-\alpha_1 x_0)(x_1-\alpha_2 x_0)\cdots (x_1-\alpha_s x_0)(y_1-\beta_1 y_0)\cdots (y_1-\beta_t y_0)$$
is the product of the $s+t$ components, and $g_0$ is a bihomogeneous polynomial of multidegree $(p,q)=(c-s,d-t)$. Let $Y=V(h_0, f)\subseteq X$ be the points covered by the $s+t$ components of $g$. We have $|Y|\le ms+nt-1$, because we are certainly double counting the point lying on the intersection of the vertical and horizontal rulings. The remaining set of points, $X\setminus Y$, must be precisely $V(g_0, f)$, whose cardinality is $aq+bp$ according to Theorem~\ref{bezout}. 

Applying Theorem~\ref{bezout} again to $f$ and $g$, it follows that 
$$a(s+q)+b(t+p) = |X| \le ms+nt-1+aq+bp.$$
Simplifying the inequality above yields
$$as+bt\le ms+nt-1.$$
Since $a\ge m$, $b\ge n$, and both $s$ and $t$ are at least one, we have a contradiction. Thus, $X$ cannot be a VCI.
\end{proof}

Using Lemma~\ref{max-deg}, we can now flesh out Remark~\ref{hyperbola} into a theorem.

\begin{thm}\label{badconfiguration}
    There exist sets of points $X_1,X_2\subset\PP^1\times \PP^1$ such that $X_1$ and $X_2$ are equivalent up to configuration, but $X_1$ is a VCI and $X_2$ is not.
\end{thm}
\begin{proof}
    Consider the configuration of six points in Figure~\ref{fig:badconfig}, and suppose it is the VCI of $f$ and $g$ with multidegree $(a,b)$ and $(c,d)$ respectively. Through the six points of Figure~\ref{fig:badconfig} there is a form of degree $(0,5)$ and a form of degree $(5,0)$. If any of the degrees were 0, say $a$, then $V(f)$ would be parallel lines, and since there are five distinct coordinates, $f$ would have degree $(0,5)$. There is no choice of $c$ and $d$ that satisfies $ad+bc=6$, so none of the degrees are 0.
    
    By Lemma~\ref{max-deg}, since $m=2$ and $n=2$, we can assume that $a\geq 2$ and $b$ or $d$ is at least $2$. But we also have $ad+bc=6$ by Theorem~\ref{bezout}, and since none of the degrees are $0$, we find that the only possibility is that the degrees of $f$ and $g$ are $(2,1)$ and $(2,2)$ (or vice versa). 
    
    Without loss of generality, let $f$ have degree $(2,1)$. By Theorem~\ref{bezout}, $f$ will intersect each vertical ruling in exactly one point unless it contains the whole ruling. Since there are two points in our configuration sharing a vertical ruling (and symmetrically a horizontal ruling), $f$ must have a degree $(1,0)$ component passing through that ruling, and therefore must have a degree $(1,1)$ component passing through the remaining four points. Conics in the projective plane (i.e. forms of degree $(1,1)$) are determined by three points, though, so this is impossible in most cases. Thus, in these cases the set of points could not be a VCI.
    
    However, in the cases where the remaining four points do lie on a conic, the points may be a VCI. For instance, letting $\kk = \CC$ for a moment, if the points have coordinates:
    \[X = \left\{ \begin{matrix}([1:1],[1:1]),([2:1],[1:2]),([3:1],[1:3]), \\
    ([4:1],[1:4]), ([1:0],[1:1]),([1:0],[1:0])\end{matrix} \right\},\]
    then $K(f,g)$ is a virtual resolution of $S/I_X$, where
    $$ f = x_0x_1y_0 - x_1^2y_1
    \;\; \text{ and } \;\; g = 24x_1^2y_0^2 - 
    x_0^2y_0y_1 - 
    50x_1^2y_0y_1+
    x_0^2y_1^2 - 
    9x_0x_1y_1^2+
    35x_1^2y_1^2.$$
\end{proof}

Theorem~\ref{cross} enables us to give a complete classification of VCI points in $\mathbb{P}^1\times\mathbb{P}^1$ that lie in a configuration that forms a Ferrers diagram. A Ferrers diagram of points in an $m$ by $n$ grid will be called a \textbf{\emph{rectangle}}. As mentioned in Section~\ref{sec:intro}, \cite{guardo_arithmetically_2015} show a set of points is arithmetically Cohen--Macaulay exactly when it forms a Ferrers diagram. The corollary below states that if $X$ is arithemetically Cohen--Macaulay, then it is a VCI if and only if it is a complete intersection.

\begin{cor}\label{ferrers}
If $X$ is a set of points in $\PP^1\times \PP^1$ forming a Ferrers diagram, then $X$ is a virtual complete intersection if and only if it is a rectangle.
\end{cor}

\begin{proof}
  Defining $m$ and $n$ as before, if $X$ forms a Ferrers diagram that is not a rectangle, then the number of points is strictly lower than $mn$. Further, the corner of the diagram is one of $m$ points on its horizontal ruling and $n$ points on its vertical ruling, so applying Theorem~\ref{cross} proves that $X$ is not a VCI.
\par The converse was proved in \cite{guardo_arithmetically_2015}*{Theorem 5.13}, but we reproduce it here for completeness. If the configuration is a rectangle, then let $f$ denote the $(m, 0)$-form whose vanishing set consists of the $m$ vertical rulings the points lie on, and let $g$ denote the $(0,n)$-form whose vanishing set consists of the $n$ horizontal rulings the points lie on. Then $(f, g)$ is a regular sequence, indicating $X$ is a complete intersection.
\end{proof}

\section{Bounds on Multidegrees and Size of A Configuration}\label{sec-nt-bounds}
In some cases, the property of being a VCI cannot be directly determined based on the maximum number of points on a single vertical/horizontal ruling. In this section, we provide characterizations for VCIs by more closely examining the relationship between the multidegrees of $f$ and $g$ and the total number of points in the configuration.

\begin{thm}\label{gcd-thm}
If $|X|<mn$ and $\gcd(m,n)$ does not divide $|X|$, then $X$ is not a VCI.
\end{thm}

Before proving Thereom~\ref{gcd-thm}, we first prove the following lemma.

\begin{lem}\label{vcideg}
Let $f$ be a bihomogeneous polynomial of multidegree $(a,b)$ and $g$ of multidegree $(c,d)$. Let $K(f,g)$ be a virtual resolution for $S/I_X$ with $|X|<mn$. By Lemma \ref{case-1}, without loss of generality, we can assume that $a\ge m$ and $b\ge n$. Then,
\begin{enumerate}
    \item $a=m$, $b=n$, and
    \item $V(g)$ has vertical components exactly on rulings with $n$ points of $X$ and has horizontal components exactly on rulings with $m$ points of $X$. 
\end{enumerate}      
\end{lem}
\begin{proof}
As in the proof of Theorem~\ref{cross}, let $s$ be the number of factors of $g$ of the form $x_1-\alpha x_0$ (i.e. the number of horizontal lines in $V(g)$) and $t$ be the number of factors of the form $y_1-\beta y_0)$ (i.e. the number of vertical lines in $V(g)$), where $\alpha, \beta$ are constants. Set $Y$ to be the points of $g$ covered by those lines, and let $(p,q)$ be the multidegree of the remaining components of $g$. By Lemma~\ref{case-1}, we have $s,t\ge 1$ and
$$as+bt=|Y|\le ms+nt.$$
By hypothesis, $a\ge m$ and  $b\ge n$. Either of these being strict would contradict the above, so $a=m$ and  $b=n$. This implies 
$$ms+nt=|Y|.$$
Thus, each vertical component of $V(g)$ must contain $n$ points of $X$ and each horizontal component must contain $m$ points of $X$, because Theorem~\ref{cross} guarantees no point can lie on both a horizontal and vertical component. 
\end{proof}

Note that when $|X|<mn$, the values of $s$ and $t$ are determined by the configuration of $X$: in this case, we must have $s=m$ and $t=n$.

\begin{proof}[Proof of Theorem~\ref{gcd-thm}]
Suppose $|X| < mn$ with $\gcd(m,n)$ not dividing $|X|$. Assume $K(f,g)$ is a virtual resolution for $S/I_X$. From Lemma~\ref{vcideg}, it can be assumed $f$ has multidegree $(m,n)$. Letting $g$ have multidegree $(c,d)$, Theorem~\ref{bezout} implies $|X|=dm+cn$. This is divisible by $\gcd(m,n)$, which is a contradiction. 
\end{proof}

\begin{prop}\label{numthry}
If $K(f,g)$ is a virtual resolution for $S/I_X$ with $|X|<mn$, $\gcd(m,n)=1$, and $(m,n)$ the multidegree of $f$, then the multidegree of $g$ is $(c,d)$, where $0\leq c < m$ and $0\leq d< n$ are unique integers satisfying:
$$c = n^{-1}|X|\mod m, \hspace{1cm} \text{ and }\hspace{1cm} d = m^{-1}|X|\mod n.$$
\end{prop}

\begin{proof}
By Theorem~\ref{bezout},
$$dm+cn=|X|.$$
Considering modulo $m$ and $n$, we have: 
$$c\equiv n^{-1}|X|\mod m \hspace{1cm} \text{and} \hspace{1cm} d\equiv m^{-1}|X|\mod n.$$
Since both $cn$ and $dm$ are less than $mn$ we must have $c<m$ and $d<n$. Thus $c$ and $d$ must have the desired values.
\end{proof}

\begin{prop}\label{num-theory-prop}
Assume $X$ is a finite set of points with $|X|<mn$, $\gcd(m,n)=1$, and let 
\[c = n^{-1}|X|\mod m \hspace{1cm} \text{and} \hspace{1cm} d = m^{-1}|X|\mod n.\]
Let $s$ and $t$ be defined as in the proof of Theorem~\ref{cross}, and set $p:=d-s$ and $q:=c-t$. If any of the following are true, $X$ will not be a VCI.
\begin{enumerate}
    \item $dm+cn \neq |X|$
    \item $d<s$ or $c<t$ 
    \item There is a horizontal ruling with strictly between $q$ and $m$ points of $X$, or a vertical ruling with strictly between $p$ and $n$ points of $X$.   
\end{enumerate}
\end{prop}

\begin{proof}
  We will prove the contrapositive of (1). Assume that $K(f,g)$ is a virtual resolution for $S/I_X$. Then by Lemma~\ref{vcideg}, $\deg(f) = (m,n)$ and by Proposition~\ref{numthry}, $\deg(g) = (c,d)$. Thus Theorem~\ref{bezout} guarantees $|X| = dm + cn$.
  
If $X$ is a VCI of $f$ and $g$, then $g$ has $s$ horizontal line components, so $d\ge s$. Similarly since it has $t$ vertical line components, then $c\ge t$. Thus (2) is proved. 
\\\\ By Lemma~\ref{vcideg}, any horizontal ruling with fewer than $m$ points of $X$ cannot be contained in $V(g)$. However, a polynomial of multidegree $(p,q)$ cannot vanish on more than $p$ points of a horizontal ruling without containing the entire ruling. Analogously, $g$ cannot vanish on between $q$ and $n$ points of a vertical ruling, completing (3).
\end{proof}
Note that when $X$ is a VCI of $f$ and $g$ with $|X|<mn$ and $\deg(f)=(m,n)$, we have determined not only the multidegree of $g$ but also the multidegree of components that are not degree 1 lines, that is, $p$ and $q$ are intrinsically determined.

\par Using these restrictions on VCIs, it is possible to classify all possible configurations of VCIs when the parameters are small and eliminate a sizable number of cases in general. Nevertheless, it is important to keep in mind Remark~\ref{hyperbola}, which illustrates the limit of the applicability of combinatorics of configurations to determine VCIs when the size of the configuration gets sufficiently large (with respect to $m$ and $n$). 

\section{Complete Classifications of Points on at Most Three Rulings}\label{sec-compl-cla}
In this section, we give a complete classification of VCIs when all points lie on at most three rulings. The case where all points lie on a single ruling is automatically a VCI (in fact, a complete intersection), and we state the conditions for 2-ruling VCIs in Theorem \ref{2-ruling}, $3$-ruling VCIs in Theorem \ref{3-ruling}.

\begin{thm}\label{2-ruling}
  If all points lie on two horizontal rulings, they form a VCI if and only if either:
  \begin{enumerate}
  \item no two of them lie on the same vertical ruling, or

  \item both horizontal rulings contain the same number points.
  \end{enumerate}
\end{thm}
\begin{proof}
If no two points lie on the same vertical ruling, then each vertical ruling contains exactly one point, so the configuration is a VCI by Theorem \ref{same-mult}. If both horizontal rulings contain the same number of points, we can match point $p_i$ in one ruling with point $p_i'$ in the other; then the configuration is the VCI of two horizontal lines through the two rulings and the product of $(1,1)$ forms each passing through one pair of $p_i$ and $p_i'$. 

Inversely, suppose two points lie on the same vertical ruling, and the two horizontal rulings have different numbers of points. Then the maximum number of points on the same horizontal ruling times 2 (the maximum number of points on the same vertical ruling) is greater than the total number of points. So by Theorem $\ref{cross}$, this configuration is not a VCI.
\end{proof}
\begin{thm}\label{3-ruling}
If all points lie on three horizontal rulings, the configuration is a VCI if and only if it satisfies one of the following conditions:
\begin{enumerate}
    \item All horizontal rulings contain the same number of points.
    \item All vertical rulings contain the same number of points.
    \item On two of the horizontal rulings all points are in pairs on the same vertical ruling, and no vertical rulings contain $3$ points.
    \item Two of the horizontal rulings contain the same number of points, $k$, and all points lie on a $(k,1)$-curve.
\end{enumerate}
\end{thm}
\begin{proof}
All four conditions are sufficient. Conditions (1) and (2) follow from Theorem \ref{same-mult}. Configurations satisfying condition (3) can always be decomposed into two rectangular blocks and be seen as an intersection of two forms that resemble the structure in Figure \ref{two-boxes}. That is, $V(f)$ consists of the two paired horizontal rulings and vertical rulings through each remaining point, and $V(g)$ consists of vertical rulings through each pair of points and a horizontal ruling through the remaining points.
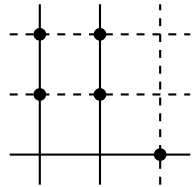
\begin{figure}[H]
\centering
\begin{tikzpicture}[scale=.8]
\foreach \x in {1,2}
	\draw[thick] (\x, 0.5) -- (\x, 3.5);
\foreach \y in {1}
	\draw[thick] (.5, \y) -- (3.5, \y);
\foreach \x in {3}
	\draw[thick,dashed] (\x, 0.5) -- (\x, 3.5);
\foreach \y in {2,3}
	\draw[thick,dashed] (.5, \y) -- (3.5, \y);
\coordinate (a) at (1,3);
\coordinate (b) at (2,3);
\coordinate (c) at (1,2);
\coordinate (d) at (2,2);
\coordinate (e) at (3,1);
\foreach \v in {a,b,c,d,e}
	\fill (\v) circle (3pt);
\end{tikzpicture}
\caption{$f$ is the product of solid lines; $g$ of dashed lines.}\label{two-boxes}
\end{figure}
Condition (4) can always be decomposed into the intersection of a $(k,1)$-curve, and the union of two horizontal lines and one vertical line through each point of $X$ on the remaining ruling, as demonstrated in Figure \ref{case_4}. By Theorem~\ref{bezout}, the $(k,1)$-curve will intersect the two horizontal rulings in exactly the $k$ points each, and will intersect each vertical ruling in exactly one point, as desired. The vertical rulings need to be carefully chosen so that their intersections with the curve lie on the same horizontal ruling.
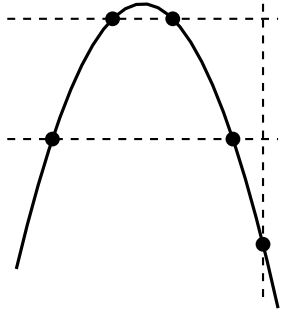
\begin{figure}[H]
\centering
\begin{tikzpicture}[yscale=.4, xscale=.4]
\myparabola{-4}{-8}{1}{-0.5}{4}{-8}
\foreach \x in {4}
	\draw[thick,dashed] (\x, 0) -- (\x, -10);
\foreach \y in {-0.5, -4.5}
	\draw[thick,dashed] (-4.5, \y) -- (4.5, \y);
\coordinate (a) at (-3,-4.5);
\coordinate (b) at (3,-4.5);
\coordinate (c) at (-1,-0.5);
\coordinate (d) at (1,-0.5);
\coordinate (e) at (4,-8);
\foreach \v in {a,b,c,d,e}
	\fill (\v) circle (7pt);
\end{tikzpicture}
\caption{$f$ is the solid parabola; $g$ is the product of dashed lines.}\label{case_4}
\end{figure}

We now show that these are the only cases. Assume $X$ is a VCI that satisfies none of the conditions above. Let $n$ be the maximum number of points on a single vertical ruling so $n$ may be $1,2$ or $3$. We will show that in each case $X$ must satisfy one of the conditions above.

Notice that if $n = 1$, then condition (2) is satisfied.
If $n=3$, one of the $3$ points on such a ruling will also be on the horizontal ruling with the maximal number $m$ of points. Theorem \ref{cross} implies $X$ will not be a VCI unless all three horizontal rulings contain $m$ points, in which case (1) holds.

We now consider the the most difficult case of $n=2$. Assume the three horizontal rulings contain $\alpha\geq \beta\geq \gamma$ points, respectively. Notice that $\alpha, \beta, \gamma$ are not all equal since otherwise $X$ satisfies condition (1). Since $X$ is a VCI, $K(f,g)$ is a virtual resolution of $S/I_X$ for some $f,g$ of multidegrees $(a,b)$ and $(c,d)$. According to Lemma \ref{max-deg}, $\max\{a,c\}\geq \alpha$ and $\max\{b,d\}\geq 2$. Without loss of generality, suppose $a\ge c$. Then $a\geq\alpha$. This means $d<3$ since $3\alpha>|X|=ad+bc\geq d \alpha$.

There are two cases:\par
\textit{Case 1:} $b\geq 2$. If $d=0$, then $V(g)$ is the union of vertical rulings. By Theorem \ref{bezout}, the number of intersection points of $V(f)$ and each vertical ruling in $V(g)$ is $b$, so the configuration must satisfy (2).

Therefore, $d>0$. Notice that $g$ having a high degree in $y$-dimension implies a low degree in $x$-dimension. In particular, degree $c\leq \beta$, because
\begin{equation}
  \alpha+2\beta\ge \alpha+\beta+\gamma=|X|=ad+bc\geq \alpha+2c \tag{6.1} \label{case1ineq}
\end{equation}
(using $a\geq\alpha$ and $b\geq 2$ in the last step).

If $c<\beta$, then $V(g)$ must contain the entirety of the horizontal rulings of $\alpha$ points and $\beta$ points.
This is because when restricting to those rulings $g$ would be a polynomial of degree $c$, and so could vanish at $\beta$ points only if it vanished on the whole ruling.
Thus, we have $d\geq 2$, and so $\alpha+\beta+\gamma=ad+bc \ge 2\alpha+2c$. 
From this we find $c<\gamma$, and again using the same argument, $V(g)$ vanishes on all 3 horizontal rulings, which violates the fact that $d<3$.
\par On the other hand, if $c=\beta$, then using the inequality (\ref{case1ineq}) and recalling $b \ge 2,d>0$, we have
\[c=\beta=\gamma, \hspace{1cm} a=\alpha, \hspace{1cm} b=2, \hspace{1cm} \text{and } \hspace{1cm}d=1.\] 
In this case, $\alpha>\beta=c$, so $V(g)$ must contain the entire horizontal ruling of $\alpha$ points. Since $d=1$, the rest of $V(g)$ can only consist of vertical rulings, and there must be $\beta$ of them. Each vertical ruling intersects $V(f)$ at 2 points by Theorem~\ref{bezout}, and so the configuration satisfies (3).

\textit{Case 2:} $b<2$. Since $\max\{b,d\}\geq 2$, $d$ must equal 2. If $b=0$, then $V(f)$ is the union of multiple vertical rulings. Since the number of intersections of $V(g)$ and each vertical ruling is always $c$, the configuration must satisfy (2). Thus, $b=1$. As before, notice that
$$\alpha+\beta+\gamma=|X|=ad+bc\geq 2\alpha+c.$$
Then, $c\leq \gamma$. If $c<\gamma$, then $V(g)$ has to contain all three horizontal rulings, which contradicts $d=2$. 
Therefore, $c=\gamma$, and $\alpha=\beta>\gamma$. In this case, $V(g)$ must contain the entire horizontal rulings of $\alpha$ points and $\beta$ points. Since $d=2$, the rest of $V(g)$ can only be vertical rulings, and there must be $\gamma$ of them, and recalling $b=1$, the configuration satisfies (4).
\end{proof}

 \raggedright

\Addresses
\end{document}